\documentclass[11pt,reqno]{amsart}
\usepackage{amssymb}
\usepackage{mathtools}
\usepackage{comment}
\usepackage{tikz-cd}
\usepackage{tikz}
\usepackage{graphicx}
\usepackage{caption}
\usepackage{subcaption}

\usepackage{booktabs}
\usepackage{float}
\usepackage{makecell}
\usepackage{thmtools}

\newif\iftodonotes
\todonotestrue   

\iftodonotes
  \usepackage{todonotes}
  \makeatletter
  \providecommand\@dotsep{5}
  \renewcommand{\listoftodos}[1][\@todonotes@todolistname]{%
    \@starttoc{tdo}{#1}}
  \makeatother
\else
  \usepackage[disable]{todonotes}
\fi

\newtheorem{theorem}{Theorem}[section]
\newtheorem{lemma}[theorem]{Lemma}

\newtheorem{proposition}[theorem]{Proposition}
\newtheorem{corollary}[theorem]{Corollary}

\theoremstyle{definition}
\newtheorem{remark}[theorem]{Remark}

\numberwithin{equation}{section}

\usepackage[backref=page]{hyperref}
\usepackage{cleveref}
\makeatletter
\@for\thm@env:=lemma,conjecture,proposition,corollary,claim,question,remark,definition,example,exercise\do{\expandafter\def\csname theH\thm@env\endcsname{\theHtheorem}}
\makeatother

\DeclareMathOperator{\rk}{rk}
\DeclareMathOperator{\Aut}{Aut}

\DeclareMathOperator{\PSL}{PSL}
\DeclareMathOperator{\GL}{GL}

\DeclareMathOperator{\SL}{SL}

\DeclareMathOperator{\Hom}{Hom}

\DeclareMathOperator{\red}{red}
\DeclareMathOperator{\gon}{gon}

\newcommand{\Q}{\mathbb{Q}}
\newcommand{\Z}{\mathbb{Z}}

\newcommand{\R}{\mathbb{R}}
\newcommand{\C}{\mathbb{C}}
\newcommand{\F}{\mathbb{F}}
\newcommand{\PP}{\mathbb{P}}

\newcommand{\Qbar}{\overline \Q}

\newcommand{\tor}{\mathrm{tors}}

\def\torz#1{\mathbb Z /#1 \mathbb Z}
\def\tg#1#2{\mathbb Z/#1\mathbb Z \times \mathbb Z /#2 \mathbb Z}
\newcommand{\diamondop}[1]{\langle #1 \rangle} 

\newcommand{\Norm}[2]{\textup{Norm}_{ #1 }( #2 ) }

\newcommand{\lmfdbec}[3]{\href{https://www.lmfdb.org/EllipticCurve/Q/#1/#2/#3}{#1.#2#3}}

\newcommand{\githubbare}[1]{\href{https://github.com/nt-lib/X0N-sporadic/blob/main/#1}{\path{#1}}}

\date{\today}
\title{Sporadic points on $X_0(N)$}
\author[Derickx]{Maarten Derickx}
\address{Maarten Derickx, University of Zagreb Faculty of Science, Department of Mathematics, Bijeni\v{c}ka cesta 30, 10000 Zagreb, Croatia}
\email{\url{maarten@mderickx.nl}}
\urladdr{\url{https://www.maartenderickx.nl/}}

\author[Najman]{Filip Najman}
\address{Filip Najman, University of Zagreb Faculty of Science, Department of Mathematics, Bijeni\v{c}ka cesta 30, 10000 Zagreb, Croatia}
\email{\url{fnajman@math.hr}}
\urladdr{\url{https://web.math.pmf.unizg.hr/~fnajman/}}
\thanks{The authors were supported by the Croatian Science Foundation under the project no.\ IP-2022-10-5008, by the project ``Implementation of cutting-edge research and its application as part of the Scientific Center of Excellence for Quantum and Complex Systems, and Representations of Lie Algebras'', PK.1.1.10.0004, co-financed by the European Union through the European Regional Development Fund -- Competitiveness and Cohesion Programme 2021--2027, and by the European Union -- NextGenerationEU through the National Recovery and Resilience Plan 2021--2026, via an institutional grant of the University of Zagreb Faculty of Science, IK IA 1.1.3, Impact4Math.}

\begin{document}

\begin{abstract}
    We classify all integers $N$ for which the modular curve $X_0(N)$ admits a sporadic CM point of arbitrary degree over $\Q$. Similarly, we determine all $N$ for which $X_0(N)$ admits a non-cuspidal sporadic point, whether CM or non-CM. In a sense, our results generalize the classification of isogenies of elliptic curves over $\Q$ due to Mazur and Kenku: their work determines the $X_0(N)$ with degree 1 sporadic points, whereas we classify all $X_0(N)$ that have a sporadic point of arbitrary degree.

\end{abstract}

\maketitle

\section{Introduction}

Following \cite{VirayVogt24}, for a nice\footnote{smooth, projective, geometrically integral} curve $X$ over a number field $k$, the \textit{density degree set} $\delta(X/k)$ is the set of integers $d$ for which $X$ has infinitely many closed points of degree $d$ over $k$, and the \textit{minimum density degree} $\min\delta(X/k)$ is the smallest element of $\delta(X/k)$. We only consider curves over $\Q$ in this paper, so we write simply $\delta(X)\coloneqq\delta(X/\Q)$. A \textit{sporadic point} $x \in X(\Qbar)$ is a point such that $\deg x< \min\delta(X)$. To avoid repetition, when speaking of sporadic points on modular curves, we will always mean (unless stated otherwise) \textit{non-cuspidal} sporadic points. 

The modular curve $X_0(N)$ is the moduli space of pairs $(E,C)$, where $E$ is a generalized elliptic curve and $C$ is a cyclic subgroup of order $N$ of $E$. It is a coarse moduli space, defined over $\Q$, and its non-cuspidal points of degree $d$ correspond to elliptic curves $E$ defined over a number field $K$ of degree $d$ together with a $K$-rational cyclic $N$-isogeny with domain $E$. Understanding the low degree points on $X_0(N)$ is therefore the same as understanding the cyclic $N$-isogenies of elliptic curves over number fields of small degree. A sporadic point is one of only finitely many points of degree at most its degree, which makes sporadic points both harder to find and more interesting.

Isolated and sporadic points on modular curves have been a subject of considerable recent interest. For a definition of isolated points, see \cite{BELOV}; the set of sporadic points is a subset of the isolated points on a curve.
In \cite{najman16} the second author found the first example of a sporadic point on $X_1(N)$ (and the one of lowest possible degree, 3), after which many have been found by van Hoeij \cite{vanHoeij}. In \cite{BELOV} Bourdon, Ejder, Liu, Odumodu and Viray  proved that, assuming Serre's uniformity conjecture, there are only finitely many $j$-invariants $j(x)\in \Q$, with $x$ a sporadic point on $X_1(N)$; this is proved unconditionally for $N$ a prime power by Bourdon and Ejder \cite{BourdonEjder25, Ejder}. 

Mazur proved in \cite{mazur77} that there are no sporadic points of degree $1$ on $X_1(N)$; the same was proved for degree $2$ by Kamienny, Kenku and Momose \cite{KM88, kamienny92}. In \cite{Deg3Class} Derickx, Etropolski, van Hoeij, Morrow and Zureick-Brown proved that the sporadic points on $X_1(21)$ found in \cite{najman16} are the only degree $3$ sporadic points on $X_1(N)$. The authors of this paper proved in \cite{DNquartic24} that there are no degree $4$ sporadic points on $X_1(N)$ and ask whether there exist sporadic points on $X_1(N)$ of every degree $d>4.$ van Hoeij's examples from \cite{vanHoeij} give sporadic points of degree $5\leq d\leq 13$ (see \cite[Appendix A]{DNquartic24}). Clark, Genao, Pollack and Saia showed in \cite{CGPS22} that all $X_1(N)$ for $N\geq 721$ have sporadic points. For further results on isolated and sporadic points on $X_1(N)$ and $X_0(N)$, see \cite{OddDeg, BHKKLMNS25,  BourdonNajman21, terao24}.

The aim of this paper is to study sporadic points on the modular curves $X_0(N)$. In \cite[Appendix~A]{DNquartic24}, the authors showed that there exist sporadic points on $X_0(N)$ (for varying~$N$) of all degrees $d \le 2166$, and conjectured that sporadic points exist on $X_0(N)$ for every degree~$d$. In this work, we investigate for which values of~$N$ the curve $X_0(N)$ admits a sporadic point, where we allow the degree of the point to vary with $N$.

Elliptic curves with complex multiplication (CM) are a common source of sporadic points on $X_0(N)$. In \cite[Theorem 8.2]{CGPS22} the authors find $50$ values of $N$ with no sporadic CM points on $X_0(N)$, 106 values for which they could not determine whether there are sporadic CM points, and prove that there are sporadic CM points for all other $N$. The first of the main results of this paper determines for the remaining 106 values of $N$ whether $X_0(N)$ has a sporadic CM point. 

\begin{theorem}\label{thm:main_CM}
    The modular curve $X_0(N)$ has a non-cuspidal sporadic CM point if and only if 
    \begin{align*}
        N\notin S_{CM}\coloneqq\{&1\text{--}10, 12, 13, 15\text{--}18, 20\text{--}26, 28\text{--}33, 35\text{--}37, 39\text{--}41, 46\text{--}50,\\
        &53, 59\text{--}61, 65, 70\text{--}72, 79, 80, 83, 87, 89, 94, 96, 101, 131, 144\}.
    \end{align*}
\end{theorem}

If a curve has a sporadic CM point, it obviously has a sporadic point. However, even for curves without a sporadic CM point, it is still possible that they have a sporadic non-CM point. Our second main result determines all modular curves $X_0(N)$ with sporadic points, whether CM or non-CM.

\begin{theorem}\label{thm:main}
    Let $N\in S_{CM}$. Then the modular curve $X_0(N)$ has a non-cuspidal sporadic point if and only if $N\in\{15,17,21,37\}$.
\end{theorem}

We note that the proof of \Cref{thm:main_CM} depends, for finitely many $N\notin S_{CM}$, on the description of $\Aut X_0(N)_\C$ in the literature, which contains several incorrect statements. See \Cref{rem:correct} for exactly which results we rely on.

Our results can, in a sense, be viewed as a generalization of the classification of isogenies of elliptic curves over $\Q$ by Mazur and Kenku (see \cite{kenku1981, mazur78}): in those papers, the modular curves $X_0(N)$ with degree-1 sporadic points are determined, whereas here we classify the curves $X_0(N)$ that have a sporadic point of arbitrary degree.

We split the proof of \Cref{thm:main_CM} by the least degree of a CM point on $X_0(N)$; the proofs are given in \Cref{sec:SporadicCM}. The reason that the authors of \cite{CGPS22} were unable to determine whether certain $X_0(N)$ have sporadic CM points is that the values of $\min\delta(X_0(N))$ were not known, and the available bounds on them were not good enough to determine whether CM points produce sporadic points. In fact, this turns out to be quite difficult. One of the tools we use here are new bounds on $\min\delta(X)$ that we obtain using recent results of Kadets and Vogt \cite{KadetsVogt25} on low-degree points on curves. These results say that for $\min\delta(X)=d$ to be relatively small, the curve has to have either a degree $d$ map to $\PP^1$ or a positive rank elliptic curve, or a very small degree map to a curve of very small genus  (see \Cref{thm:kv}). To rule out the existence of maps of degree $d$ to elliptic curves of positive rank, we find the candidate curves $E$ in the LMFDB \cite{lmfdb} and compute, following \cite[Section 2]{DerickxOrlic2024}, a positive definite quadratic form on $\Hom_\Q(J_0(N),E)$ whose values are the degrees of the morphisms $X_0(N)\rightarrow E$. To rule out degree $2$ maps from $X_0(N)$ to small genus curves, we compute the genera of the quotients of $X_0(N)$ by all its involutions as dimensions of spaces of invariants of the action on $\mathbb S_2(\Gamma_0(N))$, see \Cref{sec:genus_of_quotients}.

In \Cref{sec:lower_degree}, we prove \Cref{thm:main}. To do so, we must show that every $X_0(N)$ without sporadic CM points, apart from $N\in\{15,17,21,37\}$, also has no sporadic non-CM points. This is achieved essentially by a Mordell--Weil sieve. For most of the curves we first descend the problem to $X_0(N/2)$, which has only finitely many cubic points, and the sieve is needed only for the curves where this is not possible. Although the broad underlying idea is simple (and standard), the implementation is technically very challenging. In particular, it required determining all the degree $\leq 5$ points on $X_0(144)$, a feat that, to the best of our knowledge, has not been accomplished for $X_0(N)$ in degree greater than $3$. Indeed, more than half of the total effort on this paper was spent on various attempts to prove \Cref{144tor,prop:144}. Interestingly, the $4$ modular curves $X_0(N)$ that have sporadic non-CM points, but no sporadic CM points, have \textit{rational} non-CM sporadic points.

\subsection{Data availability and reproducibility}
All the code for the computations in this article is available at: 

\begin{center} \url{https://github.com/nt-lib/X0N-sporadic} \end{center}

The table below summarises the main computational tasks of the paper, the method used for each of them and the file in which it is carried out.

\begin{table}[H]
    \centering
    \begin{tabular}{p{0.26\linewidth}|p{0.40\linewidth}|p{0.21\linewidth}}
        Task & Method & File \\
        \hline
        Genera of the quotients of $X_0(N)$ by its involutions & the action of $\Aut X_0(N)_\C$ on $\mathbb S_2(\Gamma_0(N))$, see \Cref{sec:genus_of_quotients} & \githubbare{involutions.ipynb}, last cell \\
        \hline
        Degrees of the morphisms $X_0(N)\rightarrow E$ to elliptic curves of positive rank & a quadratic form on $\Hom_\Q(J_0(N),E)$, following \cite[Section 2]{DerickxOrlic2024} & \githubbare{quadratic_forms.py} \\
        \hline
        Cubic points on $X_0(N)$ for $N=60,70,80,94,96$ & descent to $X_0(N/2)$ and Riemann--Roch, see \Cref{prop:deg3} & \githubbare{cubic_points.m} \\
        \hline
        Cubic points on $X_0(72)$ and $X_0(87)$ & Mordell--Weil sieve modulo $5$, see \Cref{prop:deg3} & \githubbare{72.m}, \githubbare{87.m} \\
        \hline
        $J_0(144)(\Q)_\tor$ & the \'etale double cover $X_\Delta(144)$ and the methods of \cite[Section 4]{Deg3Class}, see \Cref{144tor} & \githubbare{prop_5_3.py} \\
        \hline
        Points of degree $4$ and $5$ on $X_0(144)$ & Mordell--Weil sieve modulo $5$, see \Cref{prop:144} & \githubbare{X0_144_deg_4_5.m} \\
    \end{tabular}
    \caption{The main computations of the paper.}
    \label{tab:computations}
\end{table}

The computations have been run on a server at the University of Zagreb with an Intel Xeon W-2133 CPU @ 3.60GHz with 6 cores and 64GB of RAM running Ubuntu 18.04.6 LTS using \texttt{Magma V2.28-3} \cite{magma} and \texttt{SageMath 10.6} \cite{sagemath}.

\subsection{AI use disclosure} 
No AI tools were used for the first submitted version of this paper. For this revision, we used Claude Opus 5.5 (Anthropic) to check the manuscript, the bibliography and the code, to organize the GitHub repository, and to rerun the computations. The authors take full responsibility for the content of the paper.

 \section*{Acknowledgements}
 We thank Petar Orlić for helpful comments and sharing some of his code with us, and Abbey Bourdon, Elvira Lupoian and Kenji Terao for helpful discussions.

\section{Previous work} \label{sec:previous}

Let $U$ be the set of integers $N$ for which it was not determined in \cite[Theorem 8.2]{CGPS22} whether there exists a sporadic CM point on $X_0(N)$. Following \cite{CGPS22} define $d_{CM}(X_0(N))$ to be the least degree of a CM point on $X_0(N)$. 

We split the set as 
$$U=U_4\cup U_6 \cup U_8 \cup U_{12},$$
where $U_d$ is the subset of $U$ consisting of $N$ such that $d_{CM}(X_0(N))=d$. The sets $U_d$ are as follows:

\begin{align*}
U_4\coloneqq\{&60, 70, 72, 80, 87, 90, 94, 96, 105, 108, 110, 120, 126, 132, 138, 150, 152, \\
&160, 168, 174, 188, 190, 192, 195, 204, 208, 210, 216, 230, 231, 234, 236,  \\
& 238, 240, 248, 255, 261, 264, 270, 272, 276, 282, 285, 286, 287, 303, 304,  \\& 310, 312, 315, 318, 320, 324, 332, 334, 336, 350, 357, 376, 380, 384, 392, \\
&395, 400, 413, 416, 429, 430, 435, 447, 455, 476, 483, 496, 501, 519, 524,\\
&535, 591, 623, 635\},
\\
U_6\coloneqq\{&140, 144, 180, 220, 252, 280, 288, 300, 348, 426, 432, 468, 472, 558, 560, \\
&572, 576\},\\
U_8\coloneqq\{&360, 420, 440, 504, 528, 600, 672\},\\
U_{12}\coloneqq\{&720\}.
\end{align*}

The values of $N$ where the authors of \cite{CGPS22} determine there are no sporadic CM points (see \cite[Table 2]{CGPS22}) are 
\begin{align*}
\{&1, 2, 3, 4, 5, 6, 7, 8, 9, 10, 12, 13, 15, 16, 17, 18, 20, 21, 22, 23, 24, 25, 26, 28, 29, \\
& 30,31, 32, 33, 35, 36, 37, 39, 40, 41, 46, 47, 48, 49, 50, 53, 59, 61, 65, 71, 79, 83, \\
& 89,101,131\}.
\end{align*} 
For all these values except 
\begin{equation}
N\in \{15,17,21,37\} \label{eq:spor_rat}
\end{equation}
we also know there are no sporadic points without CM on these curves, since these all either have genus $0$, or have infinitely many degree 2 points by \cite[Theorem 4.3]{Bars99} and no rational non-cuspidal points (see \cite{mazur78, kenku1981} and the references therein). The curves $X_0(N)$ for $N\in \{15,17,21,37\}$ all have rational non-cuspidal non-CM points (and finitely many of them!), hence do have sporadic non-cuspidal points. 

\section{Tools}

In this section we recall some results from other work that we will use. One of the main tools we use for the proof of \Cref{thm:main_CM} is the following result of Kadets and Vogt \cite[Theorem 1.3]{KadetsVogt25}.
\begin{theorem}\label{thm:kv}
    Suppose $X/k$ is a nice curve of genus $g$ over a number field $k$ and $\min\delta(X/k)=d.$ Let $m\coloneqq\lceil d/2 \rceil-1 $ and let $\epsilon\coloneqq3d-1-6m$. Then one of the following holds:
    \begin{itemize}
        \item[(1)] There exists a nonconstant morphism of curves $\phi\colon X\rightarrow Y$ of degree at least 2 such that $d=\min\delta(Y/k)\cdot \deg \phi$.
        \item[(2)] The genus of $X$ is bounded by 
        $$g\leq \max \left( \frac{d(d-1)}2 +1, 3m(m-1)+m\epsilon \right).$$
    \end{itemize}
\end{theorem}

\noindent In (2) we get $g\leq 7,16,33$ for $d=4,6,8$, respectively. \\

We will need \textit{Abramovich's bound} \cite[Theorem 0.1]{abramovich} (see \cite[Theorem 4.1]{DNquartic24} for the constant $\frac{325}{2^{15}}$): for a modular curve $X_H$ corresponding to a congruence subgroup $H\leq \PSL_2(\Z)$ we have
\begin{equation}\gon_\C X_{H}>\frac{325}{2^{15}}[\PSL_2(\Z):H].\end{equation}
By \cite[Proposition 2]{frey}, it follows that every $x\in X_H(\Qbar)$ that satisfies  $$\deg x< \frac{1}{2}\gon_\C X_{H}$$
is sporadic. Note that $\gon_\C X \leq \gon_\Q X$ for any curve $X/\Q$.

\begin{remark}\label{rem:gonality_bound}
We will use several times that a curve $Y$ of genus $g$ satisfies $\gon_\C Y\leq \lfloor \frac{g+3}{2}\rfloor$. This result goes back to Riemann, and a proof can be found in \cite{KleimanLaksov72} (see also \cite[Proposition A.1 (v)]{Poonen2007}).
\end{remark}

For a curve $X$ over a field $k$ and a positive integer $d$, we denote by $X^{(d)}$ the $d$-th symmetric power of $X$. Recall that the $k$-rational points of $X^{(d)}$ correspond bijectively to the effective divisors of degree $d$ on $X$ defined over $k$.

Recall that the modular degree of an elliptic curve $E$ of conductor $M$ is the minimal degree of a morphism $X_0(M)\rightarrow E$. Furthermore, if $N$ is an integer for which there exists a nonconstant map $X_0(N)\rightarrow E$, then the conductor of $E$ has to divide $N$. 

\begin{lemma}\label{lem:mod_degree}
    Let $N<778$ and $f\colon X_0(N)\rightarrow E$ be a morphism to an elliptic curve $E$ of positive rank over $\Q$. Then the conductor of $E$ divides $N$ and the modular degree of $E$ divides the degree of $f$.
    
\end{lemma}
\begin{proof}
   If $E$ has conductor $N$, this follows from the definition of the modular degree. If the conductor of $E$ is smaller than $N$, this follows from \cite[Proposition 3.6 and Proposition 3.7]{DHJO25} which were proven using the techniques of \cite[Section 2]{DerickxOrlic2024}.
\end{proof}

\subsection{Computing genera of quotients of $X_0(N)$}
\label{sec:genus_of_quotients}

The proofs that CM points on \( X_0(N) \) are sporadic in Propositions \ref{prop:deg_6_sporadic}, \ref{prop:deg_8_sporadic}, and \ref{prop:deg_12_sporadic} rely on the fact that, for \( N \in U_6 \), \( U_8 \), or \( U_{12} \), the curves \( X_0(N) \) admit no quotient by an involution of small genus. This condition is necessary in order to rule out the possibility that we are in case~(1) of \Cref{thm:kv} with a morphism of degree $2$, which is case (2) in the specializations of that theorem to the relevant degrees, as stated in Corollaries \ref{KVcor}, \ref{KVcor:8}, and \ref{KVcor:12}.

For a curve $X$ over the complex numbers we have $g(X) = \frac{1}{2} \dim  H_1(X,\Q)$. If $\iota \in \Aut X$ is an automorphism, then the genus of $X/\iota$ can be computed in terms of the action of $\iota$ on $H_1(X,\Q)$ as follows: 
$$g(X/\iota) = \frac{1}{2} \dim H_1(X/\iota,\Q) =  \frac{1}{2} \dim H_1(X,\Q)^{\iota}.$$
Note that $H_1(X/\iota,\Q)$ is naturally the module of coinvariants of the action of $\iota$ on $H_1(X,\Q)$, which is isomorphic to the module of invariants $H_1(X,\Q)^{\iota}$ by self-duality.

By the universal coefficient theorem and since $H_1(X,\Z)$ is torsion-free for complex curves, one has $H_1(X,\Q) \cong H_1(X,\Z)\otimes \Q$. If $X=X_0(N)(\C)$, then $H_1(X,\Z)$ can be described in terms of the space of modular symbols $\mathbb S_2(\Gamma_0(N))$, as originally proved in \cite[Theorem 1.9]{Manin73}. Algorithms for computing a basis of $\mathbb S_2(\Gamma_0(N))$ are described in \cite[\S 3]{stein07}, whose notation we follow. Moreover, if the automorphism $\iota$ is induced by an element $\gamma \in \GL_2^+(\Q)$ in the normalizer of $\Gamma_0(N)$, then the action of $\gamma$ on a basis of $\mathbb S_2(\Gamma_0(N);\Z)$ can be computed as described in \cite[\S 3.2]{stein07}.

Fortunately, with a few exceptions, the automorphisms of $X_0(N)$ are induced by elements $\gamma \in \GL_2^+(\Q)$. More precisely, let $\Norm{\GL_2^+(\Q)}{\Gamma_0(N)}$ denote the normaliser of $\Gamma_0(N)$ in $\GL_2^+(\Q)$ and set 
$$B_0(N)\coloneqq\Norm{\GL_2^+(\Q)}{\Gamma_0(N)}/\Q^\times\Gamma_0(N).$$ By \cite[Theorem 0.1]{KenkuMomose88} (see also \cite{Harrison11} for a correction in the case $N=108$) we have $\Aut X_0(N)_\C \simeq B_0(N)$ for $N \neq 37, 63, 108$ with $g(X_0(N))\geq 2$.

In \cite[Theorem 11]{Bars08} an explicit set of generators is given for $B_0(N)$\footnote{Bars defines $B_0(N)$ as $\Norm{\SL_2(\R)}{\Gamma_0(N)}/\Gamma_0(N)$ instead; however, up to scalars all those matrices in $\Norm{\SL_2(\R)}{\Gamma_0(N)}$ lie in $\GL_2^+(\Q)$ and all scalar matrices induce the identity automorphism of $X_0(N)$, so there is no harm in this difference.}. We compute the genera of the quotients of $X_0(N)$ by computing the action of $\Aut X_0(N)_\C$ on $\mathbb S_2(\Gamma_0(N))$ and taking invariants. This is done by the function \texttt{B0\_as\_matrix\_group} and the last cell of \githubbare{involutions.ipynb}.

\begin{remark}\label{rem:mistake}An explicit description of $B_0(N)$ in terms of generators and relations was claimed in \cite[Theorem 8]{AtkinLehner70} without proof, as a correction of \cite[Theorem 3]{LehnerNewman64}.  However, as observed in \cite{AkbasSingerman90,Bars08}, Theorem~8 of Atkin and Lehner is incorrect, since some of the relations given there do not hold. The two papers \cite{AkbasSingerman90,Bars08} independently provide corrected descriptions of \( B_0(N) \). Nevertheless, \cite{AkbasSingerman90} itself contains an inaccuracy that causes the result there to be incorrect in the subcase where \( v_2(N) = 8 \); see \cite[p.~2170]{Bars08}.

The results of \cite{Bars08} are not complete, as mentioned in Observation 17 of that article. In particular, when $v_2(N) \geq 5$ the set of relations listed is not guaranteed to be complete.

\end{remark}

\begin{remark} \label{rem:correct} As noted above, there are multiple incorrect statements in the literature regarding the structure of the automorphism group of $X_0(N)$. So we briefly explain why the main results of this article are nonetheless valid. Our reasoning relies on \cite{Harrison11} (unpublished) having addressed all errors in \cite{KenkuMomose88} and on the correctness of \cite[Theorem~11]{Bars08}. We only need the aforementioned results for the finite list of values $N$ in $(U_6 \setminus \{144\}) \cup U_8 \cup U_{12}$.

Should further mistakes in the description of \( \Aut X_0(N)_\C \) for these values of \( N \) come to light, the main results of this paper could be readily adjusted by recomputing the relevant cases with the corrected automorphism group.

\end{remark}

Remark \ref{rem:mistake} shows that there are at least four incorrect statements in the literature concerning the structure of $\Aut X_0(N)_\C$. These statements appear in the papers \cite{AkbasSingerman90, AtkinLehner70, KenkuMomose88, LehnerNewman64}, and all of these inaccuracies are such that they affect the final description of the automorphism group. Furthermore, to our knowledge, corrections to the issues in \cite{AkbasSingerman90, KenkuMomose88} have not yet appeared in print.

Therefore, it would be valuable to formalise the results on the structure of $\Aut X_0(N)_\C$ using a proof assistant. Such a formalisation would finally provide a complete and consistent account, which, at the time of writing, does not appear to be available.

\section{Sporadic CM points} \label{sec:SporadicCM}
\subsection{Degree 4}\label{sec:deg4_6}
We first deal with the easiest case, $N\in U_4$.
\begin{proposition}
    For the values $N=90$ and all $N\in U_4$ such that $N\geq 105$, $X_0(N)$ has a sporadic CM point.   
\end{proposition}
\begin{proof}
    By \cite{Jeon2024}  all the $X_0(N)$ in question have $\min\delta(X_0(N))>4.$ This same result was also independently proven in \cite[Theorem 1.8]{DerickxOrlic2024} using different methods.
\end{proof}

\begin{proposition}
    For all 
$$N\in\{60,70,72,80,87,94,96\}$$
there are no sporadic CM points on $X_0(N)$. 
\end{proposition}
\begin{proof}
    For all of these values, we have $\min\delta(X_0(N))\leq \gon_\Q X_0(N)= 4$ by \cite[Theorem 1.2]{NajmanOrlić}.
\end{proof}

\begin{remark}
    Note that we will prove a stronger result, that there are no sporadic points on these curves, of any kind, in \Cref{sec:lower_degree}.
\end{remark}

\subsection{Degree 6}

Specializing  
\Cref{thm:kv} to $d=6$
 gives a good way to rule out the existence of infinitely many degree $6$ points. This specialization to $d=6$ can also be found in \cite[Corollary 5.9]{DerickxOrlic2025}.
\begin{corollary}\label{KVcor}
If $C/\Q$ is a nice curve such that $C(\Q)\neq\emptyset$, $\min\delta(C)=6$, and $g(C) \geq 17,$ then one of the following holds:
\vskip 0.1in
(1) $C$ admits a $\Q$-rational morphism of degree $6$ to either $\PP^1$ or an elliptic curve $E$ of positive rank over $\Q$.
\vskip 0.1in
(2) $C$ admits a $\Q$-rational morphism of degree $2$ to a genus $4$ curve with infinitely many cubic points over $\Q$.
\end{corollary}
\begin{proof}
Since $g(C)\geq 17$, case (2) of \Cref{thm:kv} cannot occur. This means that there exists a non-constant morphism of curves $\phi \colon C \rightarrow Y$ of degree at least $2$ over $\Q$ with the property that $6= \min\delta(Y) \cdot \deg \phi.$ Note that $Y(\Q)\neq \emptyset$, as $C(\Q)\neq\emptyset$ and $\phi$ is defined over $\Q$.

If $\deg\phi=6$, then $Y$ has infinitely many rational points and hence it is either $\PP^1$ or an elliptic curve $E$ of positive rank over $\Q$ by Faltings' theorem.

If $\deg\phi=3$, then $\min\delta(Y)=2$, so $Y$ admits a degree $2$ rational morphism to $\PP^1$ or to an elliptic curve $E$ of positive rank over $\Q$ by \cite{harrissilverman91} and \cite[Theorem 1.2 (1)]{KadetsVogt25}. The desired map of degree $6$ is then obtained as a composition $C\to Y\to(\PP^1 \textup{ or } E)$.

If $\deg\phi=2$, then $\min\delta(Y)=3$. By \cite[Theorem 1.2 (2)]{KadetsVogt25}, either $Y$ admits a degree $3$ morphism to $\PP^1$ or to an elliptic curve of positive rank over $\Q$, in which case we again obtain the desired degree $6$ map as a composition, or $Y$ is a smooth plane quartic with no rational points, which is impossible as $Y(\Q)\neq\emptyset$, or $Y$ is a genus $4$ curve with infinitely many cubic points over $\Q$, which is case (2).
\end{proof}

\begin{proposition}\label{prop:deg_6_sporadic}
    For all $N\in U_6 \backslash \{144\}$ we have $\min\delta(X_0(N))>6$ and hence the curve $X_0(N)$ has a sporadic CM point. The modular curve $X_0(144)$ does not have a sporadic CM point.
\end{proposition}

\begin{remark}
    In fact, it turns out that $X_0(144)$ does not have any sporadic points, as we will prove in \Cref{sec:144}. It turns out that proving non-CM points cannot be sporadic is considerably more difficult than proving the same for CM points.
\end{remark}

\begin{proof}
None of them have $\min\delta(X_0(N))\leq 5$ by \cite[Theorems 1.6--1.9]{DHJO25}, so it suffices to show  $\min\delta(X_0(N))\neq 6$. We will show $\min\delta(X_0(N))\neq 6$ using Corollary \ref{KVcor}. Note that all of the curves have $g(X_0(N))\geq 19$ and none of them are of gonality $\leq 6$ over $\Q$ (see \cite[Theorems 1.2--1.4]{NajmanOrlić}). So in order to show $\min\delta(X_0(N))\neq 6$, it suffices to rule out the possibility of a map of degree 6 to a positive rank elliptic curve, or a map of degree $2$ to a genus 4 curve.

An LMFDB \cite{lmfdb} search yields a list of all elliptic curves which have positive rank, modular degree $\leq 6$ and whose conductor divides one of the integers in $U_6 \setminus \{144\}$; by Lemma \ref{lem:mod_degree} these are the only possible $E/\Q$ of positive rank to which $X_0(N)$ can have a degree $6$ map. The results of this search have been summarised in \Cref{tab:U6_pos_rank_ecs}.
\begin{table}[H]
    \centering
    \begin{tabular}{c|c|c|c}
        N   & \makecell{LMFDB\\Label} & \makecell{Modular\\degree} & Quadratic form \\
        \hline
        348 & \lmfdbec{58}{a}{1}  & 4 &
        \parbox[t]{0.5\linewidth}{
        $8(4 x_{0}^{2} - 4 x_{0} x_{1} + 4 x_{1}^{2} - 6 x_{0} x_{2} + 3 x_{1} x_{2} +$ \\$ 4 x_{2}^{2} + 3 x_{0} x_{3} - 6 x_{1} x_{3} - 4 x_{2} x_{3} + 4 x_{3}^{2})$} \\
        426 & \lmfdbec{142}{a}{1} & 4 & $8(2 x_{0}^{2} - x_{0} x_{1} + 2 x_{1}^{2})$ \\
        472 & \lmfdbec{118}{a}{1} & 4 & $8(2 x_{0}^{2} - 2 x_{0} x_{1} + 2 x_{1}^{2} + x_{0} x_{2} - 2 x_{1} x_{2} + 2 x_{2}^{2})$\\
        472 & \lmfdbec{236}{a}{1} & 6 & $12(x_{0}^{2} + x_{1}^{2})$\\
        572 & \lmfdbec{143}{a}{1} & 4 & $24(x_{0}^{2} + x_{1}^{2} - x_{0} x_{2} + x_{2}^{2})$
    \end{tabular}
    \caption{Elliptic curves whose conductor divides some $N \in U_6\backslash\{144\}$ having positive rank and modular degree $\leq 6$.}
    \label{tab:U6_pos_rank_ecs}
\end{table}

Using the methods of \cite[Section 2]{DerickxOrlic2024} we compute the quadratic forms in \Cref{tab:U6_pos_rank_ecs} such that the possible degrees of $X_0(N)\rightarrow E$ are the values of the quadratic forms in the table. The code for this computation can be found in \githubbare{quadratic_forms.py}. We see that $X_0(472)$ cannot have a degree 6 map to a positive rank elliptic curve. This proves that case (1) of Corollary \ref{KVcor} is impossible.

Case (2) in Corollary \ref{KVcor} implies that the curve has $\gon_\C X_0(N)\leq 6$, as for a curve $Y$ of genus $g(Y)=4$, we have $\gon_\C Y\leq 3$ (see \Cref{rem:gonality_bound}). For all $N\in U_6$ with $N\geq 300$, Abramovich's bound gives $\gon_\C X_0(N)\geq 8$, so this shows that we cannot be in case (2).
In the remaining cases we check that for every $N$ and every involution $\iota$ of $X_0(N)$, all of which we can explicitly obtain using the methods of \Cref{sec:genus_of_quotients}, we have $g(X_0(N)/\iota)\geq 7$, hence the case (2) in Corollary \ref{KVcor} is not possible. This is checked in the last cell of \githubbare{involutions.ipynb}.

    In the case $N=144$ we have $\min\delta(X_0(N))\leq \gon_\Q X_0(N)= 6$ by \cite[Theorem 1.4]{NajmanOrlić}.
\end{proof}

\subsection{Degree 8}
We have the following Corollary of \Cref{thm:kv} for degree 8.
\begin{corollary}\label{KVcor:8}
If $C/\Q$ is a nice curve such that $\min\delta(C)=8$ and $g(C) \geq 34,$ then one of the following holds:
\vskip 0.1in
(1) $C$ admits a $\Q$-rational morphism of degree $8$ to either $\PP^1$ or an elliptic curve $E$ of positive rank over $\Q$.
\vskip 0.1in
(2) $C$ admits a $\Q$-rational morphism of degree $2$ to a curve with infinitely many quartic points over $\Q$ of genus $\leq 7$.
\end{corollary}
\begin{proof}
Since $g(C) \geq 34$, case (2) of \Cref{thm:kv} cannot occur. This means that there exists a non-constant morphism of curves $\phi \colon C \rightarrow Y$ of degree at least $2$ over $\mathbb{Q}$ with the property that $8= \min\delta(Y) \cdot \deg \phi.$

If $\deg\phi=8$, then $Y$ has infinitely many rational points and hence it is either $\PP^1$ or an elliptic curve $E$ of positive rank over $\Q$ by Faltings' theorem.

If $\deg\phi=4$, then $Y$ has infinitely many quadratic points. Hence $Y$ admits a degree $2$ rational morphism to $\PP^1$ or to an elliptic curve $E$ of positive rank over $\Q$ by \cite{harrissilverman91} and \cite[Theorem 1.2 (1)]{KadetsVogt25}. Therefore, the desired map of degree $8$ to $\PP^1$ or $E$ can be obtained as a composition $C\to Y\to(\PP^1 \textup{ or } E)$.

If $\deg\phi=2$, then $\min\delta(Y)=4$. If $g(Y)\leq 7$ there is nothing to prove. If $g(Y) \geq 8$, we apply the same reasoning as above but this time to the curve $Y$ and $d=4$. Then again case (2) of \Cref{thm:kv} cannot occur, and by a similar argument we get a map $\phi'\colon Y \to(\PP^1 \textup{ or } E)$ of degree 4.  And hence $\phi' \circ \phi\colon C \to(\PP^1 \textup{ or } E)$ of degree $8$.
\end{proof}

\begin{proposition}\label{prop:deg_8_sporadic}
    For all $N\in U_8$ we have $\min\delta(X_0(N))>8$ and hence the curve $X_0(N)$ has a sporadic CM point. 
\end{proposition}
\begin{proof}
    We will show $\min\delta(X_0(N))>8$ for all $N\in U_8$. We first rule out $\min\delta(X_0(N))=8$ by applying Corollary \ref{KVcor:8}. We have $g(X_0(N))\geq 57$ for all $N\in U_8$. In all cases, Abramovich's bound gives us 
    $\gon_\Q X_0(N)\geq \gon_\C X_0(N)>8$. 

    We search for all elliptic curves of positive rank over $\Q$ in the LMFDB of conductor dividing $N$ and modular degree $\leq 8$; by Lemma \ref{lem:mod_degree} these are the only possible elliptic curves $E$ of positive rank such that $X_0(N)\rightarrow E$ might be of degree $8$. The curves are listed in \Cref{tab:U8_pos_rank_ecs}. The quadratic form column gives the quadratic form whose values are the possible degrees $X_0(N)\rightarrow E$, and which is computed using the methods of \cite[Section 2]{DerickxOrlic2024}. The code for this computation can be found in \githubbare{quadratic_forms.py}.
    
    \begin{table}[H]
    \centering
    \begin{tabular}{c|c|c|c}
        N   & \makecell{LMFDB\\Label} & \makecell{Modular\\degree} & Quadratic form \\
        \hline
        440 &\lmfdbec{88}{a}{1}  & 8 & $48(x_{0}^{2} - x_{0} x_{1} + x_{1}^{2})$\\
        528 &\lmfdbec{88}{a}{1}  & 8 & $32(2 x_{0}^{2} + 2 x_{1}^{2} - 3 x_{0} x_{2} + 2 x_{2}^{2} - 3 x_{1} x_{3} + 2 x_{3}^{2})$\\
        528 &\lmfdbec{176}{a}{2} & 8 & $16(2 x_{0}^{2} - x_{0} x_{1} + 2 x_{1}^{2})$\\
        600 &\lmfdbec{200}{b}{2} & 8 & $32(x_{0}^{2} - x_{0} x_{1} + x_{1}^{2})$\\
        672 &\lmfdbec{112}{a}{2} & 8 & $64(x_{0}^{2} + x_{1}^{2} - x_{0} x_{2} + x_{2}^{2} - x_{1} x_{3} + x_{3}^{2})$\\
        672 &\lmfdbec{224}{a}{2} & 8 & $32(x_{0}^{2} - x_{0} x_{1} + x_{1}^{2})$
    \end{tabular}
    \caption{Elliptic curves whose conductor divides some $N \in U_8$ having positive rank and modular degree $\leq 8$.}
    \label{tab:U8_pos_rank_ecs}
\end{table}

We see that all coefficients of the quadratic forms are divisible by 16, hence there are no maps $X_0(N)\rightarrow E$ of degree 8, showing that case (1) of Corollary \ref{KVcor:8} cannot occur.

 We note that a curve of genus $\leq 7$ has $\C$-gonality $\leq 5$, see \Cref{rem:gonality_bound}. This means that $X_0(N)$ would have $\C$-gonality $\leq 10$, see \cite[Proposition A.1 (vi)]{Poonen2007}. But for all $N \in U_8 \backslash \{360,440\}$, Abramovich's bound gives us $\gon_\C X_0(N)\geq 11.$ For the cases $N=360$ and $440$, we obtain using the methods of \Cref{sec:genus_of_quotients} that for all involutions $\iota$, we have $X_0(N)/\iota$ is of genus $\geq 25$, so we cannot be in case (2). This is checked in the last cell of \githubbare{involutions.ipynb}.

    It remains to show that $\min\delta(X_0(N))\geq 8$. By \cite[Theorems 1.6--1.9]{DHJO25} we have $\min\delta(X_0(N))\geq 6$. Suppose that $\min\delta(X_0(N))=6$. Since $g(X_0(N))\geq 57$ and $X_0(N)(\Q)\neq\emptyset$, Corollary \ref{KVcor} shows that $X_0(N)$ has a map of degree $6$ to $\PP^1$ or to an elliptic curve of positive rank, or a map of degree $2$ to a curve of genus $4$. The first is impossible as $\gon_\C X_0(N)>8$, and so is the third, as it would give $\gon_\C X_0(N)\leq 6$ by \Cref{rem:gonality_bound} and \cite[Proposition A.1 (vi)]{Poonen2007}. By Lemma \ref{lem:mod_degree}, an elliptic curve of positive rank admitting a map of degree $6$ from $X_0(N)$ would have conductor dividing $N$ and modular degree dividing $6$, but all the curves in \Cref{tab:U8_pos_rank_ecs} have modular degree $8$. Suppose now that $\min\delta(X_0(N))=7$. As $g(X_0(N))\geq 57>24$, case (2) of \Cref{thm:kv} cannot occur, so $X_0(N)$ has a map of degree $7$ to a curve $Y$ with $\min\delta(Y)=1$, i.e.~to $\PP^1$ or to an elliptic curve of positive rank, which is impossible for the same reasons. Hence $\min\delta(X_0(N))\geq 8$, and we have shown above that $\min\delta(X_0(N))\neq 8$.

\end{proof}

\subsection{Degree 12}\label{sec:12}
We need to show that the known degree $12$ CM point on $X_0(720)$ is sporadic. We have the following specialization of \Cref{thm:kv} for degree 12.

\begin{corollary}\label{KVcor:12}
If $C/\Q$ is a nice curve such that $\min\delta(C)=12$ and $g(C) \geq 86,$ then one of the following holds:
\begin{itemize}
    \item[(1)] $C$ admits a $\Q$-rational morphism of degree $12$ to either $\PP^1$ or an elliptic curve $E$ of positive rank over $\Q$.
    \item[(2)] $C$ admits a $\Q$-rational morphism of degree $2$ to a curve with infinitely many sextic points over $\Q$ of genus $\leq 16$.
    \item[(3)] $C$ admits a $\Q$-rational morphism of degree $3$ to a curve with infinitely many quartic points over $\Q$ of genus $\leq 7$.
    \item[(4)] $C$ admits a $\Q$-rational morphism of degree $4$ to a curve with infinitely many cubic points over $\Q$ of genus $\leq 4$.
\end{itemize}
\end{corollary}
\begin{proof}
This is proved using similar arguments as Corollary \ref{KVcor:8}. Unlike in Corollary \ref{KVcor}, we do not need to assume $C(\Q)\neq\emptyset$, since if the target of a degree $4$ morphism is a smooth plane quartic with no rational points, then it is a curve of genus $3$ with infinitely many cubic points, which is covered by (4).
\end{proof}

\begin{proposition}\label{prop:deg_12_sporadic}
    The curve $X_0(720)$ has a sporadic CM point. 
\end{proposition}
\begin{proof}
We need to show that $\min\delta(X_0(720))>12$. Abramovich's bound gives us $$\gon_\Q X_0(720) \geq \gon_\C X_0(720)\geq 18,$$ hence it follows that 
$$\min\delta(X_0(720))\geq \frac{\gon_\Q X_0(720)}{2}\geq 9.$$
We eliminate the possibilities of $\min\delta(X_0(720))=9,10,11,12$ by doing a case-by-case analysis.\\

\noindent $\boxed{\min\delta(X_0(720))=12}$
Since $g(X_0(720))=121$, by Corollary \ref{KVcor:12} the curve $X_0(720)$ has a degree $d$ map over $\Q$ to a curve $Y$ with infinitely many degree $12/d$ points over $\Q$, where $d\in\{2,3,4,12\}$.

\textsc{The case $d=2$}: We eliminate this case by checking that for every involution $w$ of $X_0(720)$ using the methods of \Cref{sec:genus_of_quotients}, we have $g(X_0(720)/w)\geq 57$. This is checked in the last cell of \githubbare{involutions.ipynb}.

\textsc{The case $d=3$}: by Corollary \ref{KVcor:12} (3), $Y$ would need to have $g(Y)\leq 7$. However, this would imply (see \Cref{rem:gonality_bound}) that $\gon_\C Y \leq 5$, so it would follow that $\gon_\C X_0(720)\leq 15$, which is a contradiction. 

\textsc{The case $d=4$}: by Corollary \ref{KVcor:12} (4), $Y$ would need to have $g(Y)\leq 4$. However, this would imply (see \Cref{rem:gonality_bound}) that $\gon_\C Y \leq 3$, so it would follow that $\gon_\C X_0(720)\leq 12$, which is a contradiction.

\textsc{The case $d=12$}: by Abramovich's bound, the curve cannot have a map of degree $12$ to $\PP^1$. Hence, it would need to have a map of degree $12$ to an elliptic curve with positive rank over $\Q$. A search in the LMFDB shows that there are no elliptic curves with positive rank over $\Q$ whose conductor divides $720$ and modular degree $< 16$ (this is also checked in \githubbare{quadratic_forms.py}), from which we conclude by Lemma \ref{lem:mod_degree} that there are no degree 12 morphisms $X_0(720)\rightarrow E$ to an elliptic curve with positive rank over $\Q$.

\noindent $\boxed{\min\delta(X_0(720))=11}$ Applying \Cref{thm:kv} to $\min\delta=11$ and $X_0(720)$, we get that the only possibility is a degree $11$ map to $\PP^1$ or a positive rank elliptic curve, both of which we have seen are  impossible. 

\noindent $\boxed{\min\delta(X_0(720))=10}$ Applying \Cref{thm:kv} to $\min\delta=10$ and $X_0(720)$ tells us that we need to rule out degree $d=2,5,10$ maps to curves with infinitely many degree $10/d$ points. All of these degrees are easily eliminated using the same arguments as in the $\min\delta=12$ case. 

\noindent $\boxed{\min\delta(X_0(720))=9}$ This case is ruled out exactly the same as the previous cases. 

\end{proof}

\section{Ruling out lower degree points}\label{sec:lower_degree}
It remains to prove \Cref{thm:main}. To do this, we need to check whether there are any sporadic non-CM points on the modular curves $X_0(N)$ for 
$$N\in \{60,70,72,80,87,94,96,144\}.\label{rem_cases}$$
For all $N\neq 144$, this amounts to showing that there are no non-cuspidal points on $X_0(N)$ of degree $\leq 3$, while for $N=144$ we need to rule out the possibility of points of degree $\leq 5$. 

We prove the nonexistence of degree $2$ and $3$ points on these curves in the following 2 propositions. 
\begin{proposition}\label{prop:deg2}
    For all 
$$N\in\{60,70,72,80,87,94,96,144\}$$
there are no non-cuspidal quadratic points on $X_0(N)$. 
\end{proposition}
\begin{proof}
The possibility of non-cuspidal quadratic points on $X_0(N)$ has been ruled out for $N=60$ in \cite[Section 9.1]{NajmanVukorepa23}, for $N=94$ in \cite[Theorem 3.2 (d)]{NajmanVukorepa23}, for $N=70$ and $87$ in \cite[Tables 4 and 9]{NajmanNovak2025}, for $N=96$ in \cite[Proposition 4.3]{NajmanNovak2025}, for $N=72$ in \cite[Table 8.13]{OzmanSiksek19}, and  for $N=80$ in \cite[Table 6]{AKMNOV24}. The fact that $X_0(72)$ has no non-cuspidal quadratic points immediately implies that neither does $X_0(144)$.
\end{proof}

\begin{proposition}\label{prop:deg3}
    There are no non-cuspidal points of degree $3$ on $X_0(N)$ for 
    $$N\in \{60,70,72,80,87,94,96,144\}.$$
\end{proposition}
\begin{proof}
    First note that $J_0(N)(\Q)$ is finite for all the $N$; see \cite[Theorem 3.1]{Deg3Class}. In the cases $N=60,70,80,94,96$, the curve $X_0(N/2)$ has only finitely many cubic points by \cite[Theorem 0.1]{Jeon2021}. Let $n\coloneqq N/2$. Let $C\in X_0(n)(\Q)$ be a rational cusp. Any rational effective divisor $D$ of degree $3$ satisfies $D\sim A+3C$ for a divisor $A$ representing a class in $J_0(n)(\Q)$, which is generated by the cusps by \cite[Proposition 3]{BruinNajman15}. If $l(A+3C)=1$, then $D$ is the unique effective divisor in its class. If $l(A+3C)\geq 2$, then, since $X_0(n)$ is hyperelliptic of genus $3$ or $4$, the linear system $|A+3C|$ consists of the divisors $P+F$ with $F$ in the hyperelliptic pencil and $P$ a fixed rational point, so it contains no cubic points. Hence by looping through all $A\in J_0(n)(\Q)$ with $l(A+3C)=1$ we find all the cubic points on $X_0(n)$. We check for each point whether it lifts to a cubic point on $X_0(N)$. None do, so we are done in these cases. The \texttt{Magma} code that checks this can be found at \githubbare{cubic_points.m}.

    We now consider the cases $N=72$ and $87$. Let $J_C(N)$ denote the cuspidal subgroup of $J_0(N)$, i.e.~the subgroup generated by divisors supported on cusps. We compute that $J_C(87)(\Q)\simeq\tg{14}{140}$ and by considering the reductions of $J_0(87)$ modulo $5,7,11$ and using the fact that reduction modulo an odd prime of good reduction is injective on $J_0(N)_\tor$, we conclude $J_C(87)(\Q)= J_0(87)(\Q)$. The fact that $J_C(72)(\Q)= J_0(72)(\Q)$ has been proved in \cite[Theorem 1.1]{Lupoian24}\footnote{There is a gap in the proof of this in the published version, which has been corrected in the arXiv version \url{https://arxiv.org/abs/2205.13017}.}, and we have 
    $$J_C(72)(\Q)\simeq \tg{2}{4} \times \tg{12}{12}.$$
    
    Let $D$ be a hypothetical rational non-cuspidal effective divisor on $X_0(N)$ of degree 3, and let $C\in X_0(N)(\Q)$ be a cusp.
    The curve $X_0(N)$ is not trigonal over $\F_5$ (see \cite[Section 6]{DerickxNajman24}), and the gonality does not increase under specialization by \cite[Lemma 5.3]{CadoretTamagawa16}, so $X_0(N)$ is not trigonal over $\Q$ either. Hence the map $X_0(N)^{(3)}(\Q)\rightarrow J_0(N)(\Q)$ sending $D$ to $[D-3C]$ is injective. Moreover, every effective degree 3 divisor $B$ over $\F_5$ satisfies $l(B)=1$. 
    
    Since $\rk J_0(N)(\Q)=0$ and $X_0(N)$ has good reduction at $5$, the reduction mod $5$ map $J_0(N)(\Q) \rightarrow J_0(N)(\F_5)$ is injective. Hence the composition map
    $$f\colon X_0(N)^{(3)}(\Q) \rightarrow J_0(N)(\F_5), \quad B\mapsto [(B-3C)_{\F_5}]$$
    is injective as well. Thus, there is at most one effective rational divisor of degree 3 reducing to any fixed $\F_5$-divisor class. 
    
    Cusps of $X_0(N)(\overline \Q)$ reduce injectively modulo any prime above 5 to $X_0(N)(\overline \F_5)$, so different effective rational cuspidal divisors of degree $3$ map 
    under $f$ to different divisor classes in $J_0(N)(\F_5)$. For each $x$ in the image of $J_0(N)(\Q)$ in $J_0(N)(\F_5)$, we choose a representative divisor $A_x$ such that $[A_x]=x$, and count the number of such $x$ that have $l(A_x+3C_{\F_5})=1$. If the number of such $x$ equals the number of effective rational degree-3 cuspidal divisors, then there are no non-cuspidal effective rational divisors of degree $3$ on $X_0(N)$.

    The modular curve $X_0(87)_\Q$ has $4$ cusps, all of which are defined over $\Q$. We compute the $l(A_x+3C_{\F_5})$ for all $A_x\in J_C(87)(\F_5)$, and find exactly $20$
    divisors $x\in J_C(87)(\F_5)$ such that $l(A_x+3C_{\F_5})=1$. This matches the number of rational effective degree 3 cuspidal divisors:
    $$\binom{\#\text{cusps}+3-1}{3}=\binom{6}{3}=20.$$
    Hence all effective rational divisors of degree $3$ on $X_0(87)$ are cuspidal. The code verifying this can be found in \githubbare{87.m}.

    We do a similar search for $X_0(72)$ and find $152$ divisors $A_x\in J_C(72)(\F_5)$ such that $l(A_x+3C_{\F_5})=1$. The curve $X_0(72)_\Q$ has $8$ rational cusps and $4$ degree $2$ cusps, giving 
    $$\binom{8+3-1}{3}+4\cdot 8=152$$
    cuspidal divisors. Thus every effective rational divisor of degree $3$ on $X_0(72)$ is cuspidal. The code verifying this can be found in \githubbare{72.m}.

    Finally, since $X_0(72)$ has no degree $\leq 3$ non-cuspidal points, neither does $X_0(144).$
\end{proof}

\subsection{Degree $4$ and $5$ points on $X_0(144)$}\label{sec:144}

Proving that $X_0(144)$ has no sporadic points turns out to be by far the hardest case.

\begin{proposition}\label{144tor}
The group $J_0(144)(\Q)$ is generated by divisors supported on the cusps and is isomorphic to $\torz{4}^3 \times \torz{24}^3.$
\end{proposition}

\begin{proof}
The Mordell--Weil rank of $J_0(144)(\Q)$ is 0 by \cite[Theorem 3.1]{Deg3Class}, so $J_0(144)(\Q)=J_0(144)(\Q)_{\tor}$ and hence we are reduced to computing the torsion subgroup of $J_0(144)$.
As before, $J_C(144)$ denotes the cuspidal subgroup of $J_0(144)$, i.e.~the subgroup generated by divisors supported on cusps.  We compute $J_C(144)(\Q)\simeq \torz{4}^3 \times \torz{24}^3$ using the methods of \cite[Section 4]{Deg3Class}, which compute the cuspidal subgroup of a modular Jacobian and bound its rational torsion using the reductions modulo several primes of good reduction and the Eichler--Shimura relation, so it suffices to show $J_C(144)(\Q)=J_0(144)(\Q)_\tor$. 

We tried to apply the methods of \cite[Section 4]{Deg3Class} and \cite{Lupoian24} in order to show $J_C(144)(\Q)=J_0(144)(\Q)_\tor$ directly but all of these failed. Our strategy will be to find a cover $X \to X_0(144)$ so that we can compute $J(X)(\Q)_\tor$ and use this cover to compute $J_0(144)(\Q)_\tor$.

To this end let $\Delta=\langle -1, 11 \rangle \subseteq (\torz{144})^\times$. Then $f\colon X_\Delta(144) \to X_0(144)$ is a map of degree $2$. The diamond operator $\iota \coloneqq \langle 5\rangle$ is the involution of $X_\Delta(144)$ such that $X_0(144) = X_\Delta/\iota$. The genus of $X_\Delta(144)$ is $25$, so, by the Riemann--Hurwitz formula, the map $f$ is unramified and hence \'etale. In particular, $X_\Delta$ is a subcover of the Shimura cover $X_2(N) \to X_0(N)$ from \cite[Equation II.11.5]{mazur77}. Let $f^*\colon J_0(144)\rightarrow J_\Delta(144)$ be the pullback map. Similarly to \cite[Proposition II.11.6]{mazur77}, we get that $\ker f^*$ is isomorphic to the Cartier dual of $\torz{N}^\times/\Delta$, i.e.~we have the following isomorphisms of group schemes over $\Q$: $$\ker f^* \cong \Hom(\torz{N}^\times/\Delta, \mathbb G_m) \cong \Hom(\torz{2}, \mathbb G_m) \cong \torz{2}.$$  

Using the methods of \cite[Section 4]{Deg3Class} we compute that $J_\Delta(144)(\Q)_\tor$ is generated by cuspidal divisors and that $$J_\Delta(144)(\Q)_\tor \cong \torz{2}^2 \times \torz{6} \times \torz{24}^4 \times \torz{48}.$$
Now $f^*(J_0(144)(\Q))$ is invariant under the action of $\diamondop{5}$ so most certainly $f^*(J_0(144)(\Q)) \subseteq J_\Delta(144)(\Q)^{\diamondop{5}}.$
We compute that $J_\Delta(144)(\Q)^{\diamondop{5}} \cong \torz{2} \times \torz{4}^2 \times \torz{24}^3$. From this we get that $J_0(144)(\Q)_\tor$ fits in the exact sequence
$$ \torz{2} \to J_0(144)(\Q)_\tor \to J_\Delta(144)(\Q)^{\diamondop{5}} \cong \torz{2} \times \torz{4}^2 \times \torz{24}^3.$$

However, since we already know that  $J_0(144)(\Q)_\tor$ contains a subgroup isomorphic to $\torz{4}^3 \times \torz{24}^3$, this implies that $J_0(144)(\Q)_\tor$ has to actually be isomorphic to $\torz{4}^3 \times \torz{24}^3$ for cardinality reasons. All the computations in this proof were done in \texttt{Sage} using the package \texttt{mdsage} of the first author, which implements the methods of \cite[Section 4]{Deg3Class}; the code can be found in \githubbare{prop_5_3.py}.

\end{proof}

\begin{remark}
    While trying to prove the previous proposition, we noticed a curious phenomenon. The natural map $f\colon X_0(144)\rightarrow X_0(16)$ is of degree $12$. Taking $c\in X_0(16)(\Q)$, we get that $f^*(c)$ is a theta characteristic, i.e.~$2f^*(c)$ is in the canonical class. The chance of a rational divisor class of degree $12$ chosen at random being a theta characteristic is $\frac{\#J_0(144)(\Q)[2]}{\#J_0(144)(\Q)_\tor}=\frac{1}{13824}$. We note that the choice of $c\in X_0(16)(\Q)$ does not affect the linear equivalence class of $f^*(c)$, as $X_0(16)\simeq \PP^1$, so $f^*(c)$ will be linearly equivalent to $f^*(d)$ for any other $d\in X_0(16)(\Q)$.
\end{remark}

\begin{proposition}\label{prop:144}
    The modular curve $X_0(144)$ has no non-cuspidal points of degree $4$ or $5$.
\end{proposition}
\begin{proof}

    We will show that any $D\in X_0(144)^{(d)}(\Q)$ is cuspidal for $d=4,5$. We have the following commutative diagram
\begin{center}
    \begin{tikzcd}
X_0(144)^{(d)}(\Q) \arrow[r,hook, "A"'] \arrow[d,hook, "\red_5"] &J_0(144)(\Q)   \arrow[d,hook,"\red_{5,J}"] \\
X_0(144)^{(d)}(\F_5) \arrow[r,hook, "A_5"']               & J_0(144)(\F_5)
\end{tikzcd},
\end{center}
where $A$ and $A_5$ are the Abel--Jacobi maps (with some fixed base point), and the vertical maps are reductions modulo $5$. The maps $A$ and $A_5$ are injective because $\gon_{\F_5}X_0(144)=\gon_{\Q}X_0(144)=6>d$ by \cite[Theorem 1.4 and Proposition 5.15]{NajmanOrlić}, while $\red_{5,J}$ is injective by the injectivity of the torsion and $\rk J_0(144)(\Q)=0$. The injectivity of $\red_5$ follows by the commutativity of the diagram.

Any $D\in X_0(144)^{(d)}(\Q)$ has to satisfy
$$A_5 \circ \red_{5}(D)\in \red_{5,J}(J_0(144)(\Q)).$$ We run through all $x\in X_{0}(144)^{(d)}(\F_5)$ and count the number of those satisfying $A_5(x)\in \red_{5,J}(J_0(144)(\Q))$. We obtain that their number is exactly equal to the number of effective rational cuspidal divisors of degree $d$ on $X_0(144)$. The \texttt{Magma} computations that were used to verify this can be found in \githubbare{X0_144_deg_4_5.m}. From the injectivity of the maps in the diagram, it follows that $D$ is cuspidal. 
\end{proof}

\appendix
\section{Castelnuovo--Severi inequality with at least 3 maps}

When considering how to bound the gonalities of some of the curves in this paper, we proved the following result, which can be used to bound the gonality of a curve $X$ using the quotient by an automorphism of prime order, e.g.~an involution. We did not need this result in the end, but we believe it might be of use to others, and of independent interest.

\begin{proposition}\label{prop:CSstrong}
    Let $X$ be a curve of genus $g$ over a perfect field $k$, let $\iota \in \Aut_k X$ be an automorphism of prime order $p$ and let $G\coloneqq\diamondop{h \cdot \iota \cdot h^{-1} \mid h \in \Aut_k X}$. If $f\colon X\rightarrow \PP^1_k$ is a map of degree $d$ defined over $k$ such that
    \begin{equation} \label{eq:CS}
        g>p\cdot g(X/\iota) + (p-1)(d-1),
    \end{equation}
    then $f$ factors through $X/G$. In particular $\#G$ divides $d$ and $d\geq \#G\cdot \gon_k (X/G)$.
\end{proposition}
\begin{proof}
    First note that if \eqref{eq:CS} is satisfied, then by the Castelnuovo--Severi inequality (see \cite[Theorem 14]{KhawajaSiksek24}) both $f$ and the quotient map $X\rightarrow X/\iota$ factor through a map $X\rightarrow X'$ of degree $>1$ defined over $k$. Since $X\rightarrow X/\iota$ has prime degree $p$, the map $X'\rightarrow X/\iota$ has degree $1$, and is hence an isomorphism, so $f$ factors through $X/\iota$. For an automorphism $h\in \Aut_k X$, the automorphism $h \cdot \iota \cdot h^{-1}$ is also of order $p$ and the curves $X/\iota$ and $X/h \cdot \iota \cdot h^{-1}$ are isomorphic, so $X/h \cdot \iota \cdot h^{-1}$ also satisfies \eqref{eq:CS}. Hence $f$ also factors through $X/h \cdot \iota \cdot h^{-1}$. The proposition now follows.
\end{proof}
Note that if $\diamondop{\iota}$ is a normal subgroup of $\Aut_k X$ (e.g.~if $\iota$ is in the center of $\Aut_k X$), then $G=\diamondop{\iota}$ and \Cref{prop:CSstrong} is just the Castelnuovo--Severi inequality. The real use lies in the case when $G$ is much larger than $\diamondop{\iota}$.

\bibliographystyle{amsalpha}
\bibliography{bibliography1}
\end{document}